\documentclass[11pt]{amsart}

\usepackage{txfonts}
\newcommand{\nonarc}{\multimapinv}
\usepackage{lmodern}
\usepackage{amscd,amssymb,graphicx,url}
\usepackage{color}
\usepackage{soul} 

\begin{document}

\newtheorem{thm}{Theorem}
\newtheorem*{thm*}{Theorem}
\newtheorem{cor}[thm]{Corollary}
\newtheorem{lemma}[thm]{Lemma}
\newtheorem{prop}[thm]{Proposition}

\theoremstyle{definition}
\newtheorem{quest}[thm]{Question}
\newtheorem{ex}[thm]{Example}

\newcommand{\linnum}{\stepcounter{thm}\tag{\thethm}}

\newcommand{\bN}{\mathbb N}
\newcommand{\bR}{\mathbb R}
\newcommand{\cA}{\mathcal{A}}
\newcommand{\cB}{\mathcal{B}}
\newcommand{\cC}{\mathcal{C}}
\newcommand{\cD}{\mathcal{D}}
\newcommand{\cE}{\mathcal{E}}
\newcommand{\cF}{\mathcal{F}}
\newcommand{\cJ}{\mathcal{J}}
\newcommand{\cM}{\mathcal{M}}
\newcommand{\cN}{\mathcal{N}}
\newcommand{\cO}{\mathcal{O}}
\newcommand{\cR}{\mathcal{R}}
\newcommand{\cS}{\mathcal{S}}
\newcommand{\cU}{\mathcal{U}}
\newcommand{\cV}{\mathcal{V}}
\newcommand{\cW}{\mathcal{W}}
\newcommand{\cY}{\mathcal{Y}}
\newcommand{\tx}{\tilde{x}}
\newcommand{\tv}{\tilde{v}}
\newcommand{\wtD}{{\widetilde{D}}}
\newcommand{\wtgamma}{\widetilde{\gamma}}
\newcommand{\wtphi}{\widetilde{\phi}}
\newcommand{\wtU}{{\widetilde{U}}}
\newcommand{\wtV}{{\widetilde{V}}}
\newcommand{\za}{\alpha}
\newcommand{\zb}{\beta}
\newcommand{\zf}{\phi}
\newcommand{\zF}{\Phi}
\newcommand{\zg}{\gamma}
\newcommand{\zd}{\delta}
\newcommand{\ze}{\epsilon}
\newcommand{\zh}{\eta}
\newcommand{\zj}{\psi}
\newcommand{\zl}{\lambda}
\newcommand{\zm}{\mu}
\newcommand{\zn}{\nu}
\newcommand{\zp}{\pi}
\newcommand{\zq}{\theta}
\newcommand{\zr}{\rho}
\newcommand{\zs}{\sigma}
\newcommand{\zS}{\Sigma}
\newcommand{\zt}{\tau}
\newcommand{\zv}{\varphi}
\newcommand{\zw}{\omega}
\newcommand{\zD}{\Delta}
\newcommand{\zJ}{\Psi}
\newcommand{\zG}{\Gamma}
\newcommand{\zL}{\Lambda}
\newcommand{\lcm}{\mbox{lcm}}
\newcommand{\co}{\colon\thinspace}
\newcommand{\SR}{S_{\mathcal{R}}}
\newcommand{\SRp}{S_{\mathcal{R'}}}
\newcommand{\expm}{\varphi}
\newcommand{\subm}{\sigma_{\mathcal{R}}}
\newcommand{\defn}{\emph}
\newcommand{\Chat}{\widehat{\mathbb C}}
\newcommand{\Otilde}{\widetilde{\mathcal{O}}}
\newcommand{\bbC}{{\mathbb C}}
\newcommand{\bbD}{{\mathbb D}}
\newcommand{\bbR}{{\mathbb R}}

\newcommand{\gap}{\vspace{5pt}}      
\newcommand{\mtwo}[4]                
{\mbox{$\left(\begin{array}{cc}      
#1 & #2 \\
#3 & #4
\end{array}
\right)$}}

\newcommand{\kmp}[1]{\sf \color{blue} #1 \rm \color{black}}
\newcommand{\nonslope}{\odot}
\newcommand{\pullback}{\stackrel{\,f}{\leftarrow}}

\newcommand{\wrp}[1]{\textcolor{red}{\sf #1}}

\newcommand{\pf}{\noindent {\bf Proof: }}

\newcommand{\id}{\mbox{\rm id}}
\newcommand{\interior}{\mbox{int}}   
\newcommand{\diam}{\mbox{\rm diam}}
\newcommand{\dist}{\mbox{\rm dist}}
\newcommand{\per}{\mbox{\rm per}}
\newcommand{\qedspecial}[1]{\nopagebreak \begin{flushright}
        \rule{2mm}{2.5mm}{\bf #1} \end{flushright}}

\newcommand{\nosubsections}{\renewcommand{\thethm}{\thesection.\arabic{thm}}
           \setcounter{thm}{0}}

\newcommand{\bdry}{\partial}

\title{Expansion properties for finite subdivision rules II}

\author{William Floyd}
\address{Department of Mathematics\\ Virginia Tech\\
Blacksburg, VA 24061\\ U.S.A.}
\email{floyd@math.vt.edu}
\urladdr{http://www.math.vt.edu/people/floyd}

\author{Walter Parry}
\address{Department of Mathematics and Statistics\\ 
Eastern Michigan University\\
Ypsilanti, MI 48197\\ U.S.A.}
\email{walter.parry@emich.edu}

\author{Kevin M. Pilgrim}
\address{Department of Mathematics\\ Indiana University\\
Bloomington, IN 47405\\ U.S.A.}
\email{pilgrim@indiana.edu}

\keywords{finite subdivision rule, expanding map, postcritically
finite, Thurston map}
\subjclass[2000]{Primary 37F10, 52C20; Secondary 57M12}
\date\today

\begin{abstract} We prove that every sufficiently large iterate of a
Thurston map which is not doubly covered by a torus endomorphism and
which does not have a Levy cycle is isotopic to the subdivision map of
a finite subdivision rule.  We determine which Thurston maps doubly
covered by a torus endomorphism have iterates that are isotopic to 
subdivision maps of finite subdivision rules.  
We give conditions under which no iterate
of a given Thurston map is isotopic to the subdivision map of a finite
subdivision rule.
\end{abstract}
\maketitle

\section{Introduction}\label{sec:intro}

This paper continues our study of expansion properties for finite
subdivision rules begun in \cite{exppropi}.  In \cite{exppropi} we
usually began with a Thurston map which is the subdivision map of a
finite subdivision rule.  Here we begin with a general Thurston map,
and we ask whether it is isotopic to the subdivision map of a finite
subdivision rule. 

An answer to the question of which Thurston maps are isotopic to
subdivision maps of finite subdivision rules remains out of reach.  A
much more tractable problem is to determine which Thurston maps have
\emph{iterates} which are isotopic to subdivision maps of finite
subdivision rules.  Passage to an iterate does not materially affect
dynamics, so in terms of dynamics, passage to an iterate is not very
restrictive.

Our main result, Theorem~\ref{thm:main}, is that every sufficiently
large iterate of a Thurston map which is not doubly covered by a torus
endomorphism and which does not have a Levy cycle is isotopic to the
subdivision map of a finite subdivision rule.  A key ingredient in our
proof of this is a result \cite[Theorem C]{bd0} of Bartholdi and
Dudko, which implies that every Thurston map which is not doubly
covered by a torus endomorphism and which does not have a Levy cycle
is B\"{o}ttcher expanding.  In our proof of Theorem~\ref{thm:main}, 
what we really use is the fact that the
map is B\"{o}ttcher expanding.  In Theorem~\ref{thm:euclid} we handle
those Thurston maps which are doubly covered by torus endomorphisms.
So the question for iterates is answered except for Thurston maps
which are not covered by torus endomorphisms and which have Levy
cycles.  Of course, many of these are subdivision maps of finite
subdivision rules.  However, in Theorem~\ref{thm:cnds} we present a
condition on such a Thurston map $f$ which implies that no iterate of
$f$ is isotopic to the subdivision map of a finite subdivision rule.
Section~\ref{sec:nofsr} concludes with an example of such a Thurston
map.

We conclude this introduction with a brief history of this problem. There are two sources of motivation: one from geometry, another from dynamics.  On the geometric side, finite subdivision rules were introduced in \cite{fsr} as
simplifications of patterns (informally described as 'local replacement rules' or 'finite replacement rules') observed when studying the action of cocompact Kleinian groups on the Riemann sphere;  see e.g. \cite[\S 3]{cfp:shape}. 
On the dynamical side, given a topological dynamical system, it is natural to ask whether there are Markov partitions with good topological properties.  Thus
the problem naturally arises to determine which
Thurston maps are subdivision maps of finite subdivision rules.
Independent solutions of this problem for all sufficiently large
iterates of postcritically finite rational maps with no periodic
critical points (they all are) were obtained almost simultaneously by
Bonk-Meyer and Cannon-Floyd-Parry.  The latter solution appeared as
Theorem 1 in \cite{subrat}.  The former solution appeared 
as Corollary 15.2 in \cite{BM}, where the result was generalized to
expanding Thurston maps.  The main result of \cite{lattes} is that
almost every Latt\`{e}s rational map is the subdivision map of a
finite subdivision rule with one tile type of a very special form.  In
\cite[Theorem 1.2]{ghmz} Gao, Ha\"issinsky, Meyer and Zeng prove that
every sufficiently large iterate of a postcritically finite rational
map whose Julia set is a Sierpi\'{n}ski carpet is the subdivision map
of a finite subdivision rule.  More recently Cui, Gao, and Zeng \cite{cgz2019} proved that any critically finite rational map 
has an iterate that preserves a finite connected graph containing the postcritical set, and thus is the subdivision map of a finite subdivision rule. Thus although the present work applies
to more Thurston maps, its conclusion is weaker than those above
because its conclusion only gives that the iterate is isotopic to a map which is a subdivision map.
Thus it would be interesting to know if the techniques of \cite{cgz2019} can be adapted to B\"ottcher expanding maps.

\section{Realizability of B\"{o}ttcher expanding maps}
\label{sec:bottcher}\nosubsections

The goal of this section is to prove the main theorem,
Theorem~\ref{thm:main}.  To prepare for Theorem~\ref{thm:main}, we
will fix some notation, make some definitions and prove four lemmas.

The notation and definitions which we introduce now will hold
throughout this section.  As in Theorem~\ref{thm:main}, let $f\co
S^2\to{S^2}$ be a Thurston map which is not doubly covered by a torus
endomorphism and which does not have a Levy cycle.  

\subsection*{B\"ottcher expanding maps.}  Let $P_f$ be the postcritical set 
of $f$, and let $P_f^\infty $ be the set of periodic points in $P_f$ whose 
cycles contain a critical point.  Following  \cite[Definition 4.1]{bdiv} we 
say that $f$ is \emph{B\"ottcher expanding} if (i) there exists a complete 
length metric on $S^2-P_f^\infty$  which is expanded by $f$, in the sense 
that for every nonconstant rectifiable compact curve $\gamma: [0,1] \to 
S^2-P_f^\infty$, the length of any lift of $\gamma$ under $f$ is strictly 
less than the length of $\gamma$, and (ii) for each $p \in P_f^\infty$, 
if $n$ is the least period of $p$, then the first-return map $f^n$, near $p$, 
is locally holomorphically conjugate to $z \mapsto z^k$ near the origin, 
where $k=\deg(f^n, p)$.  Any two different choices of such local holomorphic 
coordinates differ by a multiplicative constant which is a $(k-1)$st root of 
unity. 

For example, if $f$ is rational, then according to \cite{DH} there is
a canonically associated orbifold structure on $S^2$ whose singular
points are precisely the elements of $P_f$, with elements of
$P_f^\infty$ having weight infinity, so that these points correspond to
punctures.  The constant curvature metric on the orbifold universal
cover descends to a metric which is then expanded by $f$; see
\cite[Appendix A]{Mc}.

Our point of departure for our subsequent analysis is the following result.
As usual, the case of maps $f$ whose orbifold $\cO_f$ is Euclidean (parabolic) require special treatment. 
Among these, those with three or fewer postcritical points are isotopic to rational maps, by Thurston's characterization \cite{DH}, and these are B\"ottcher expanding. The remaining types -- those whose orbifolds have signature $(2,2,2,2)$ -- are the exceptions in our setting. Any such map $f$ is isotopic through maps agreeing on $P_f$ to a map $g$ which is affine in the natural Euclidean orbifold structure, and which lifts under the natural orbifold double-cover to an affine map of a torus induced by an affine map of the plane of the form $x \mapsto A_fx + b$.

The characterization in case 1 below is a deep result of Bartholdi and Dudko,  announced in \cite[Theorem C]{bd0} and proved in \cite[Theorem A=Theorem 4.4]{bdiv}.

\begin{thm}[Expanding conditions]
\label{thm:expanding_conditions}
Suppose $f$ is a Thurston map.
\begin{enumerate}
\item If $\cO_f$ does not have signature $(2,2,2,2)$, then $f$ is isotopic to a smooth B\"ottcher expanding map $g$ if
and only if it has no Levy cycles.
\item If $\cO_f$ has signature $(2,2,2,2)$, then $f$ is isotopic to a smooth B\"ottcher expanding map $g$  if and only if the eigenvalues of $A_f$ lie outside the closed unit disk. 
\end{enumerate}
\end{thm}

The characterization in case 2 is elementary.  (Sufficiency is clear. To prove necessity, note that an eigenvalue inside the closed unit disk must be real; a curve given by a segment in the corresponding eigenspace and containing a fixed-point will have inverse images which do not get shorter under backward iteration, and this is an obstruction to the existence of some length metric which is contracted under taking preimages.)

So in the remainder of this section we may, and do, assume that $f$ itself is smooth and B\"{o}ttcher expanding.  

We will need uniform expansion. To accomplish this, we will excise
open forward-invariant neighborhoods, round-shaped in the natural
conjugating coordinates, of points in $P_f^\infty$ to obtain a compact
planar subset $\cM$ with the property that its preimage $f^{-1}(\cM)$ is
contained in the interior of $\cM$. From compactness and expansion it
then follows that $f$ is uniformly expanding on $\cM$ in the sense
that for some $0<\rho<1$ and for every nonconstant rectifiable compact
curve $\gamma: [0,1] \to \cM$, the length of any lift of $\gamma$
under $f$ is at most $\rho$ times the length of $\gamma$.

\subsection*{Fatou and Julia sets} 

Suppose $f$ is a B\"ottcher expanding smooth Thurston map, and $\cM$
is the complement of neighborhoods of the attractors as in the
previous subsection.  We recall here some facts from \cite[Section
4.2]{bdiv}.

The \emph{Julia set} $\cJ_f$ is the closure of the set of repelling periodic 
points; a periodic point is \emph{repelling} if there is no neighborhood 
$U\ne S^2$ with $f^k(U)$ compactly contained in $U$ for some $k>0$.  
The \emph{Fatou set} $\cF_f$  is the set of points at which the map 
$S^2\ni z \mapsto (z, f(z), f^2(z), \ldots) \in (S^2)^\infty$ is continuous. 
By \cite[Lemma 4.6]{bdiv}, $S^2=\cJ_f \sqcup \cF_f$, and an equivalent 
formulation of $\cJ_f$ is that it coincides with the set of points which 
do not converge under iteration to the attracting cycles in $P_f^\infty$, 
i.e. with the set of points which never escape $\cM$. 

\noindent{\bf Notation.}  Let $\bbD$ be the open unit disk in $\bbC$.
We use an overline to denote the topological closure of a set.  So
$\overline{\bbD}$ is the closed unit disk. An \emph{arc} is the image
of a closed unit interval under a homeomorphism; arcs are always
closed.  Let $d\co S^2\times S^2\to \bbR$ be a spherical metric.  All
distances considered between points in $S^2$ are relative to $d$.

For each $p\in P_f^\infty$, the immediate basin of attraction  $F_p$ of
$p$ is the component of the Fatou set that contains $p$. 
Then for each $p\in P_f^\infty$, there is a homeomorphism 
$\zj_p\co \bbD\to F_p$ with the following property.  If $p\in P_f^\infty$, 
$q=f(p)$ and $k$ is the local degree of $f$ at $p$, then 
$\zj_q^{-1}(f(\zj_p(z)))=z^k$ for every $z\in \bbD$. Moreover,
the Douady-Hubbard proof \cite[Section 3.5]{DH2} 
of the local connectivity of the filled Julia set of a
subhyperbolic polynomial carries over in this setting and so each $\zj_p$ 
extends to a continuous map $\zj_p\co \overline{\bbD}\to \overline{F}_p$.
For more information, see the discussion before Lemma 4.7 in \cite{bdiv}
or Theorem 4.14 of Milnor's paper \cite{m2}.

The images in $\overline{F}_p$ of radii in $\overline{\bbD}$ are
called \emph{rays}, and the images of circles in $\bbD$ concentric
about 0 are called \emph{equipotentials}; these are independent of the
choice of local holomorphic coordinates. The ray $\{\zj_p(r e^{i
\zq}):0\le r\le 1\}$ with angle $\zq$ is said to \emph{land} at
$\zj_p(e^{i \zq})$.  For every nonnegative integer $n$, the connected
components of $f^{-n}(F_p)$ are called \emph{Fatou components}.  Just
as for rational Thurston maps which are B\"ottcher expanding, the
Julia set of $f$ is connected, and each Fatou component is
homeomorphic to a disk and is eventually periodic.  Because $f$ is
B\"{o}ttcher expanding, there exists a finite cover of $S^2\setminus
\bigcup_{p\in P_f^\infty}F_p$ by open sets $U_i$ such that the
diameters of the connected components of the sets $f^{-n}(U_i)$ tend
to 0 as $n\to \infty $.

\subsection*{Proof of Theorem \ref{thm:main}}

The initial groundwork for the proof of Theorem~\ref{thm:main} is
complete.  Now we continue with four lemmas.  The first lemma is a
result in combinatorial topology about connecting vertices in a tiling
of the disk. The next three lemmas are finiteness results about points
in the boundary of an immediate basin of attraction.  We emphasize
that the above notation and definitions remain in force throughout
this section.

The first lemma is essentially Lemma 4 in \cite{subrat}, though in
\cite{subrat} we assume that the CW complex gives a regular tiling. In
the setting of a CW complex with underlying space a surface, the
closed $2$-cells are called \emph{tiles}.  The assumption in the second sentence is nontrivial, since open $2$-cells need not be Jordan domains, and so a tile might contain a vertex in its interior.

\begin{lemma}\label{lemma:tileddisk} Let $X$ be a closed topological
disk with the structure of a CW complex such that every tile of $X$ is
a closed topological disk.  Let $u_1$, $u_2$ and $v$ be any triple of distinct
vertices of $X$ such that each is in the boundary of a tile of $X$.
Then there is an arc in the 1-skeleton of $X$ which has initial point
$u_1$, terminal point $u_2$ and contains $v$.
\end{lemma}
  \begin{proof} We prove Lemma~\ref{lemma:tileddisk} by induction of
the number of tiles in $X$.  If $X$ has only one tile, then $u_1$,
$u_2$ and $v$ are on the boundary of a polygon, and the theorem is
clear.  Suppose that $n$ is a positive integer and that the result
holds if $X$ has $n$ tiles.  Let $Y$ be a closed topological disk
which has the structure of a CW complex with $n+1$ tiles such that
every tile of $Y$ is a closed topological disk.  Let $u_1$, $u_2$, $v$
be distinct vertices of $Y$ such that each is in the boundary of a
tile of $Y$.  Because every tile of $Y$ is a closed topological disk,
$Y$ has distinct tiles $t_1$ and $t_2$ such that the closures of
$Y\setminus t_1$ and $Y\setminus t_2$ are topological disks.  Let $t$
be one of these tiles such that $v$ is in the closure $X$ of
$Y\setminus t$.  If $u_1,u_2\in X$, then we are done by induction.
Now suppose that $u_1\in X$ but $u_2\notin X$.  Let $w\ne u_1$ be one
of the two points of $t\cap X\cap \partial Y$.  By induction there is
an arc in $X$ with endpoints $u_1$ and $w$ which contains $v$.  This
arc can easily be extended to $u_2$.  The case in which $u_1\notin X$
and $u_2\in X$ is symmetric, and the case in which $u_1,u_2\notin X$
is similar.

\end{proof}

The second lemma shows that there cannot be infinitely many rays
landing at a preperiodic point in the boundary of an immediate basin
of attraction.

\begin{lemma}\label{lemma:land} Let $p\in P_f^\infty $, and let $q$ be
a preperiodic point in $\partial F_p$.  Then only finitely many rays
in $\overline{F}_p$ land at $q$.
\end{lemma}
  \begin{proof} By replacing $f$ by an iterate of $f$, we may assume
that $f(p)=p$ and $f(q)=q$.  Let $\cR$ be the set of rays in $\overline{F}_p$ which
land at $q$.

We now define an equivalence relation $\sim $ on $\cR$.  Let
$R_1,R_2\in \cR$.  If $R_1=R_2$, then $R_1\sim R_2$.  Suppose that
$R_1\ne R_2$.  Then $R_1\cup R_2$ is a simple closed curve.  We set
$R_1\sim R_2$ if one of the connected components of $S^2\setminus
(R_1\cup R_2)$ contains no element of $P_f$.  This defines an
equivalence relation on $\cR$.  Because $P_f$ is finite, this
equivalence relation has only finitely many equivalence classes
$\cR_1,\dotsc,\cR_j$.  

If $P_f$ contains just two points, then $f$ is Thurston equivalent to
$z\mapsto z^{\deg(f)}$ and the lemma is true in this case.  So we
assume that $P_f$ contains a point other than $p$ and $q$.  Hence if
$R,S\in \cR$ with $R\ne S$ and $R\sim S$, then exactly one connected
component of $S^2\setminus (R\cup S)$ contains no element of $P_f$.

Let $R\in \cR$.  Then $f^n(R)\in \cR$ for every nonnegative integer
$n$.  If the $j+1$ rays $R$, $f(R)$, $f^2(R),\dotsc,f^j(R)$ are not
distinct for every $R\in \cR$, then $\cR$ is finite.  So we may assume
that $R$ is chosen so that these $j+1$ rays are distinct.

One of the equivalence classes $\cR_1,\dotsc,\cR_j$ must contain at
least two of these $j+1$ rays.  We may assume that this equivalence
class is $\cR_1$.  So there exist $R_1,R_2\in \cR_1$ with $R_1\ne R_2$
such that some iterate of $f$ takes $R_2$ to $R_1$.  We replace $f$ by
this iterate and assume that $f(R_2)=R_1$.

Now let $D$ be the closed topological disk bounded by $R_1\cup R_2$
whose interior contains no element of $P_f$.  The interior of $D$ is
evenly covered by $f$.  Since $f(R_2)=R_1$, the disk $D$ therefore has
a lift to a disk $\widetilde{D}$ whose boundary is the union of $R_2$
and a ray $R_3$ in $\overline{F}_p$ which lands at $q$.  Then
$f(R_3)=R_2$ and $R_3\in \cR_1$ because the interior of
$\widetilde{D}$ contains no element of $P_f$.  Furthermore, the
closure of the connected component of $S^2\setminus (R_1\cup R_3)$
which contains no element of $P_f$ contains $R_2$.

This construction can be repeated.  We obtain an infinite sequence
$R_1,R_2,R_3,\ldots $ of rays in $\cR_1$ such that $f(R_{n+1})=R_n$
for every positive integer $n$.  Furthermore, the closure of the
connected component of $S^2\setminus (R_1\cup R_n)$ which contains no
element of $P_f$ contains $R_2,\dotsc,R_{n-1}$ for every integer $n\ge
3$.  It follows that the rays $R_1,R_2,R_3,\ldots $ converge to a ray
in $\cR$ fixed by $f$.  The proof of Lemma 18.12 in Milnor's book
\cite{m1} finally implies that if $\cR$ contains a ray fixed by $f$,
then only finitely many rays in $\overline{F}_p$ land at $q$.

This completes the proof of Lemma~\ref{lemma:land}.

\end{proof}

The third lemma is about how many points in the boundaries of two
immediate basins of attraction are accessible to a point that is in
neither of their closures.

\begin{lemma}\label{lemma:onebasin} Let $p$ and $q$ be distinct
elements of $P_f^\infty $, and let $x\in S^2\setminus
(\overline{F}_p\cup \overline{F}_q)$.  Then there exist at most two
points $y\in \partial F_p\cap \partial F_q$ for which there exists an
arc $\zg$ in $S^2$ with endpoints $x$ and $y$ such that $\zg\cap
(\overline{F}_p\cup \overline{F}_q)=\{y\}$.
\end{lemma}
  \begin{proof} Suppose that there are three distinct such points
$y_1$, $y_2$, $y_3$.  Let $i\in\{1,2,3\}$.  Let $\zr_i$ be a ray in
$\overline{F}_p$ which joins $p$ and $y_i$, and let $\zs_i$ be a ray
in $\overline{F}_q$ which joins $q$ and $y_i$.  Then $S^2\setminus
(\bigcup_{j=1}^{3}(\zr_j\cup \zs_j))$ has three connected components.
So the set $\bigcup_{j=1}^{3}(\zr_j\cup \zs_j)$ separates $x$ from one
of the points $y_1$, $y_2$, $y_3$.

This contradiction proves Lemma~\ref{lemma:onebasin}.

\end{proof}

Here is the fourth lemma.  Figure~\ref{fig:canisotop} illustrates one
of the finitely many situations in question.  The region shaded by
line segments is a connected component of the complement of the union
of $\zr_x$, $\zr_{y_\ze}$, $T_\ze$ and $\zg\setminus \{z\}$.  Its
interior contains a point $p'\in P_f$.

  \begin{figure}
\centerline{\includegraphics{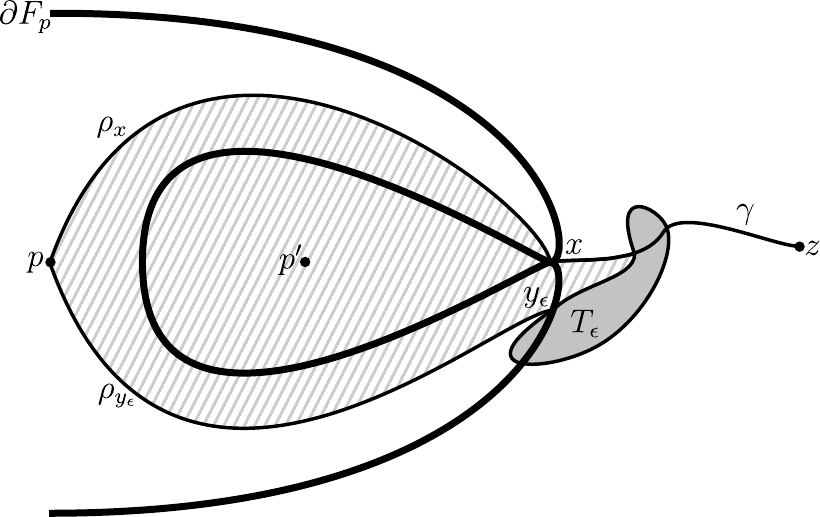}}
\caption{The situation in Lemma~\ref{lemma:canisotop}. The point $p'$ is the solid black dot in the middle.}
\label{fig:canisotop}
  \end{figure}

\begin{lemma}\label{lemma:canisotop} Let $p\in P_f^\infty $, and let
$C$ be a connected component of $S^2\setminus \overline{F}_p$.  Then
there exist only finitely many points $x\in\partial F_p$ for which the
following conditions are satisfied.
\begin{enumerate}
  \item There is a nontrivial arc $\zg$ in $S^2$ with one endpoint
$x$ whose other points are in $C$.  Let $z$ be the endpoint of $\zg$
other than $x$.
  \item There exist real numbers $\ze>0$ accumulating at 0 such that
for each of them there exists a path connected set $T_\ze$ in
$S^2\setminus \{z\}$ whose diameter is less than $\ze$ containing a
point $y_\ze\in \partial F_p$ and a point of $\zg$.
  \item There exist rays $\zr_x$ and $\zr_{y_\ze}$ in $\overline{F}_p$
which land at $x$ and $y_\ze$ such that the union of $\zr_x$,
$\zr_{y_\ze}$, $T_\ze$ and $\zg\setminus \{z\}$ separates $z$ from an
element of $P_f$.
\end{enumerate}
\end{lemma}
  \begin{proof} Let $x$ be a point in $\partial F_p$ which satisfies
the conditions of the lemma.  The compactness of $S^2$, $\zg$ and
$\overline{F}_p$ and condition 2 imply that a sequence of the sets
$\overline{T_\ze}$ converges in the Hausdorff topology to a point in
$\zg\cap \partial F_p$.  This point must be $x$.  Using condition 3,
we see that by passing to a subsequence we may assume that the
corresponding sequence of rays $\zr_{y_\ze}$ converges to a ray
$\zr'_x$ which lands at $x$ such that $\zr_x\cup \zr'_x$ is a simple
closed curve which separates $C$ from an element of $P_f$.  These
simple closed curves separate $C$ from different elements of $P_f$ for
different values of $x$.  Since $P_f$ is finite, there are only
finitely many such points $x$.

This proves Lemma~\ref{lemma:canisotop}.

\end{proof}

We are now ready to prove the main theorem, Theorem~\ref{thm:main}.

\begin{thm}\label{thm:main} Let $f$ be a Thurston map such that (i) if
it does not lift to a torus endomorphism, it has no Levy cycles and
(ii) if it lifts to a torus endomorphism, the associated affine map
has eigenvalues outside the unit circle.  Then every sufficiently
large iterate of $f$ is isotopic to the subdivision map of a finite
subdivision rule.
\end{thm}

Before giving the proof, we discuss the strategy of the proof and the
dependence of the constants that arise in the proof.  The first thing
we do is identify three subsets $P_f^\infty$, $P_f^1$, and $P_f^0$
that form a partition of the postcritical set $P_f$.  The idea is to
construct a graph, $G\subseteq S^2$, invariant up to isotopy, and with vertex set containing $P_f$.  For
this we construct three finite sets $\cA$, $\cB$, $\cC$ of curves in $S^2$. The sets $\cB, \cC$ consist of arcs, while $\cA$ consists of closed curves which are either simple or trivial (a point); the latter are the elements of $P_f^0$.
Our first approximation of $G$ as a set is the union of the curves in
$\cA \cup \cB \cup \cC$.

Our first constant $\rho\in (0,1)$ is chosen sufficiently close to $1$
based on an open cover $\{U_i\}$ of $S^2\setminus \bigcup_{p\in
P_f^\infty}F_p$ that exists because $f$ is B\"ottcher expanding.  Once
we have $\rho$, we construct the collection $\cA$ of elements of
$P_f^0$ and of equipotential curves corresponding to radius $\rho$,
one for every $p\in P_f^\infty $. The set $\cB$ consists of tails of
rays.  For every $p\in P_f^\infty$ the set $\cB$ contains the tails
from the equipotential of $F_p$ in $\cA$ to an element $q\in P_f^1
\cap \partial F_p$.

Our second constant $\delta$ is a sufficiently small positive real
number such that $d(p,x) > 4\zd$ if $p\in P_f^0$ and $x$ is either in
$P_f^0$ with $x\ne p$ or $x$ is in the closure of an immediate basin
of attraction.  The set $\cC$ is a finite collection of arcs joining
curves in $\cA$ which are carefully chosen to satisfy ten conditions;
the second and third conditions depend on $\zd$.  

Having constructed the sets $\cA$, $\cB$ and $\cC$, we give $S^2$ the
structure of a cell complex $S$ whose 1-skeleton is the union of the
curves in $\cA$, $\cB$ and $\cC$.  We obtain a cell complex
$f^{-n}(S)$ for every nonnegative integer $n$.  Tiles in $f^{-n}(S)$
which are mapped by some iterate of $f$ into immediate basins of
attraction are called Fatou tiles.  The other tiles of $f^{-n}(S)$ are
Julia tiles.

We next choose a sufficiently small real number $\ze\in (0,\zd)$. The
number $\ze$ is chosen to give certain disjointness properties of
$\ze$-neighborhoods of the arcs in $\cC$, and to also satisfy one
additional constraint coming from one of the ten conditions. Having
chosen $\ze$, we use the B\"{o}ttcher expanding property to choose a
sufficiently large positive integer $N$ such that if $n\ge N$, then
every Julia tile in $F^{-n}(S)$ has diameter less that $\ze$.

Now suppose that $n$ is a positive integer with $n\ge N$.  We enlarge
$\cA \cup \cB \cup \cC$ to get the final form of $G$ by adding
equipotential curves and radial arcs in the immediate basins of
attraction. We add enough of these so that every Fatou tile in
$f^{-n}(S)$ has diameter less than $\ze$. Finally, we show that $f^n$
is isotopic rel $P_f$ to the subdivision map of a finite subdivision
rule whose 1-skeleton is $G$.

  \begin{proof} We partition the postcritical set $P_f$ of $f$ into
three subsets.  The subset $P_f^\infty $ consists of those periodic
elements of $P_f$ whose cycles contain critical points.  We let
  \begin{equation*}
P_f^1=\{p\in P_f:p\in \partial F_q,q\in P_f^\infty
\}\quad\text{and}\quad
P_f^0=\{p\in P_f:p\notin \bigcup_{q\in P_f^\infty }\overline{F}_q\}.
  \end{equation*}

We intend to construct a finite subdivision rule $\cR$.  This involves
equipping $S^2$ with the structure of a CW complex $S_\cR$.  The
complex $S_\cR$ will be defined as a subdivision of a CW complex $S$.
We will define $S$ by means of its 1-skeleton, which is a union of
curves.  These curves belong to three finite sets of curves $\cA$,
$\cB$ and $\cC$.  The curves in $\cA$ are either trivial, just points,
or simple closed curves.  The curves in $\cB$ and $\cC$ are nontrivial
arcs.

We construct the set of curves $\cA$ in this paragraph.  It depends on
a real parameter $\zr$ with $0<\zr<1$.  Let $\bbD_\zr=\{z\in
\bbC:\left|z\right|<\zr\}$.  For every $p\in P_f^\infty $ let
$D_p=\zj_p(\bbD_\zr)\subseteq F_p$.  Recall that we have open sets
$U_i$ whose existence is a consequence of the fact that $f$ is
B\"{o}ttcher expanding.  We choose and fix $\zr$ so that the open
disks $D_p$ together with the open sets $U_i$ cover $S^2$.  The simple
closed curves in $\cA$ are those curves of the form $\partial D_p$ for
$p\in P_f^\infty $.  The trivial curves in $\cA$ are the points in
$P_f^0$.  This completes the construction of $\cA$.

Now we construct the set $\cB$.  If $P_f^1=\emptyset $, then
$\cB=\emptyset$.  Suppose that $P_f^1\ne \emptyset $.  Let $p\in
P_f^1$.  Then there exists $q\in P_f^\infty $ such that $p\in \partial
F_q$.  For every such point $q$ and for every ray $R$ in
$\overline{F}_q$ which lands at $p$, we require that $\cB$ contains
the subarc of $R$ which has one endpoint $p$ and whose other endpoint
is in $\partial D_q$.  This completes the construction of $\cB$.
Lemma~\ref{lemma:land} implies that $\cB$ is a finite set of arcs.

We next construct the set $\cC$.  For this we choose a real number
$\zd>0$ such that $d(p,x)>4\zd$ if $p\in P_f^0$ and $x\in P_f^0\cup
(\bigcup_{q\in P_f^\infty }\overline{F}_q)$ with $x\ne p$.  For every
$x\in S^2$, let $B_r(x)$ be the open ball of radius $r>0$ about $x$.
The set $\cC$ will be constructed as a set of arcs $\zg$ which satisfy
the following conditions.

\begin{enumerate}
  \item The endpoints of every arc $\zg$ are in distinct curves in
$\cA$, and one of these endpoints is in $\bigcup_{p\in P_f^\infty
}^{}\partial D_p$ if $P_f^\infty \ne \emptyset $.
  \item If $p\in P_f^0 $ and $p\in \zg$, then $\zg\cap \partial
B_\zd(p)$ and $\zg\cap \partial B_{2\zd}(p)$ each contain exactly one
point.
  \item If $p\in P_f^0$ and $p\notin \zg$, then $\zg\cap
\overline{B}_{2\zd}(p)=\emptyset $.
  \item If $p\in P_f^\infty $ and $\zg\cap \partial D_p\ne \emptyset
$, then $\zg\cap \overline{F}_p$ is a subarc of a ray.
  \item If $p\in P_f^\infty $ and $\zg\cap \partial D_p\ne \emptyset
$, then $\zg$ avoids the finitely many points in $\partial F_p$ which
satisfy the conditions of Lemma~\ref{lemma:canisotop} relative to the
connected component of $S^2\setminus \overline{F}_p$ which contains
$\zg\setminus \overline{F}_p$.
  \item If $p\in P_f^\infty $ and $\zg\cap \partial D_p=\emptyset $,
then $\zg\cap \overline{F}_p=\emptyset $.
  \item If $\zb\in \cB$, then $\zg\cap \zb=\emptyset $.
  \item Two such arcs $\zg$ may meet only at a point of $P_f^0$.
  \item Every element of $P_f^0$ is an endpoint of exactly two of the
arcs $\zg$.
  \item The connected components of the complement in $S^2$ of the
union of the curves in $\cA$, $\cB$ and $\cC$ are open topological
disks whose closures are closed topological disks.
\end{enumerate}

Let $G$ be the union of the curves in $\cA$ and $\cB$.  As we
construct arcs in $\cC$, we adjoin them to $G$.  So $G$ will grow
during this construction.  It will eventually become a graph.

If $P_f^\infty =\emptyset $, then $G=P_f=P_f^0$.  In this case we
construct the arcs in $\cC$ so that their union $G$ is a simple closed
curve.  It is easy to do this and satisfy conditions 1--10.

So suppose that $P_f^\infty \ne \emptyset $.  Figure~\ref{fig:gamma}
contains a schematic diagram of one possibility for $G$ after the
construction of $\cC$ is complete.  The large dots are in $P_f^0$.
The small dots are in $P_f^1$.  The circles are the simple closed
curves in $\cA$.  The line segments are the arcs in $\cB$.  The
remaining arcs are the arcs in $\cC$ which we are about to construct.

  \begin{figure}
\centerline{\includegraphics{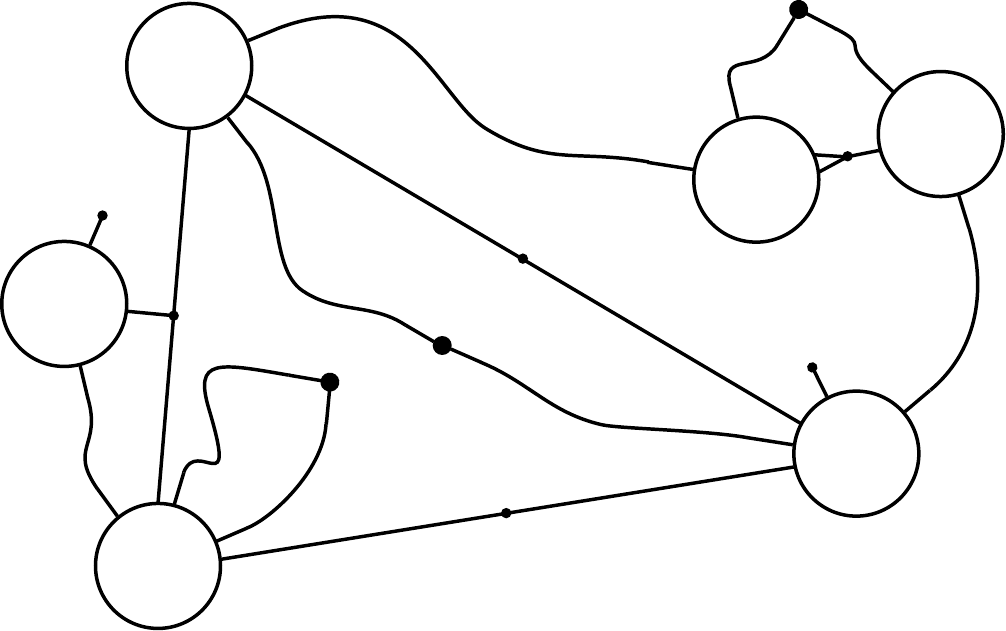}}
\caption{The union of the curves in $\cA$, $\cB$ and $\cC$ }
\label{fig:gamma}
  \end{figure}

Let $p\in P_f^0$.  It is easy to construct an arc $\za$ with one
endpoint $p$, one endpoint in $\bigcup_{q\in P_f^\infty
}\partial D_q$, disjoint from $P_f\setminus \{p\}$ and satisfying
conditions 2, 3, 5 and 7.  The arc $\za$ has a minimal subarc $\zg$
which joins $p$ and $\bigcup_{q\in P_f^\infty }\overline{F}_q$.
Lemma~\ref{lemma:onebasin} implies that by choosing $\za$ to avoid
finitely many points, we may assume that the endpoint of $\zg$ other
than $p$ lies in just one of the sets $\partial F_q$ with $q\in
P_f^\infty $.  We then extend $\zg$ along a ray in $\overline{F}_q$ to
$\partial D_q$.  Now $\zg$ satisfies conditions 1--7.  We construct two such
arcs $\zg$ for every $p\in P_f^0$ so as to satisfy conditions 1--9.
We put these arcs in $\cC$ and adjoin them to $G$.  After doing this
for every element of $P_f^0$, every connected component of $G$
contains $\partial D_p$ for some $p\in P_f^\infty $.

Suppose that there exist $p,q\in P_f^\infty $ such that $\partial D_p$
and $\partial D_q$ lie in different connected components of $G$.  We
partition $P_f^\infty $ into two disjoint nonempty subsets $Q_1$ and
$Q_2$ such that if $p\in Q_1$ and $q\in Q_2$, then $\partial D_p$ and
$\partial D_q$ lie in different connected components of $G$.  Let
$H_i=\bigcup_{q\in Q_i}\overline{F}_q$ for $i\in \{1,2\}$.

Suppose that $H_1\cap H_2$ is an infinite set.  Then there exist $p\in
H_1$ and $q\in H_2$ such that $\overline{F}_p\cap \overline{F}_q$ is
an infinite set.  Now we repeatedly apply Lemma~\ref{lemma:onebasin}
to $p$, $q$ and $x\in P_f^\infty \setminus \{p,q\}$ to conclude that
there exists $y\in \partial F_p\cap \partial F_q$ such that $y\notin
\partial F_{p'}\cup P_f^1$.  We choose a ray in $\overline{F}_p$ with endpoint $y$ and a
ray in $\overline{F}_q$ with endpoint $y$.  The union of these rays
contains a subarc $\zg$ with endpoints in $\partial D_p$ and $\partial
D_q$.  It is possible to choose $y$ so that $\zg$ satisfies condition
5.  Therefore conditions 1--9 are maintained by putting $\zg$ in
$\cC$.  We put $\zg$ in $\cC$ and adjoin it to $G$.  This reduces the
number of connected components of $G$.

Suppose that $H_1\cap H_2$ is a finite set.  Then we construct an arc
$\za$ which joins $H_1$ and $H_2$ and which avoids $H_1\cap H_2$, all
arcs in $\cB\cup \cC$ already constructed and all closed balls
$\overline{B}_{2\zd}(p)$ with $p\in P_f^0$.  The arc $\za$ has a
minimal subarc $\zg$ which joins $\overline{F}_p$ and $\overline{F}_q$
with $p\in Q_1$ and $q\in Q_2$.  We may choose $\za$ so that $\zg$
satisfies condition 5.  Using Lemma~\ref{lemma:onebasin}, we find that
$\za$ may also be chosen so that $p$ and $q$ are unique.  We extend
$\zg$ along rays in $\overline{F}_p$ and $\overline{F}_q$ to an arc
with endpoints in $\partial D_p$ and $\partial D_q$ while satisfying
conditions 1--9.  We put this arc in $\cC$ and adjoin it to $G$.  This
reduces the number of connected components of $G$.

After constructing finitely many arcs as above, the set $G$ is
connected.  So the connected components of $S^2\setminus G$ are open
topological disks.  Let $D$ be one of these disks, and suppose that
$\overline{D}$ is not a closed topological disk.  Then there exists a
simple closed curve $\za$ in $\overline{D}$ with exactly one point not
in $D$ which satisfies the following property.  There exist $p,q\in
P_f^\infty $ such that $\partial D_p\cap \overline{D}\ne \emptyset $, 
$\partial D_q\cap \overline{D}\ne \emptyset $ and $\za$ separates 
$\partial D_p$ from $\partial D_q$.
Working as in the previous paragraph, we construct an arc $\zg$ with
interior in $D$ satisfying conditions 1--9 which joins $\partial D_{p'}$ and
$\partial D_{q'}$, where $p'$ and $q'$ are elements of $P_f^\infty $ such as
$p$ and $q$.

In this way, we can eventually satisfy every condition 1--10.  This
completes the construction of $\cC$.

We now have $\cA$, $\cB$ and $\cC$.  The union $G$ of their curves
contains $P_f^0$ and $P_f^1$ but not $P_f^\infty $.  We make $G$ into
a 1-dimensional CW complex (graph) in the straightforward way: we
declare that its vertices are the endpoints of the arcs in $\cB\cup
\cC$.  This in turn equips $S^2$ with the structure of a 2-dimensional
CW complex $S$.

For every nonnegative integer $n$ we use $f^n$ to pull back $S$ to
obtain a CW complex $f^{-n}(S)$.  We call the closed 2-cells of such
complexes tiles.  We say that a tile is a Fatou tile if it is
contained in some Fatou component, equivalently, if its image under
some iterate of $f$ is contained in $F_p$ for some $p\in P_f^\infty $.
Tiles which are not Fatou tiles are Julia tiles.

We next choose a real number $\ze$ with $0<\ze<\zd$.  Conditions 2 and
3 imply that every arc $\zg\in \cC$ intersects $S^2\setminus
\bigcup_{p\in P_f^0}B_{\zd}(p)$ in a subarc $\zg^*$.  We choose an
open topological disk $D_\zg\subseteq S^2\setminus P_f$ which contains
$\zg^*$ such that $D_\zg\cap D_{\zg'}=\emptyset $ if $\zg,\zg'\in \cC$
with $\zg\ne \zg'$.  We also require that i) $D_\zg\cap
\overline{F}_p=\emptyset $ if $p\in P_f^\infty $ such that $\zg\cap
\overline{F}_p=\emptyset $ and ii) $D_\zg\cap B_{2\zd}(p)=\emptyset $
if $p\in P_f^0$ such that $p\notin \zg$.  We choose $\ze$ so that
$D_\zg$ contains the $\ze$-neighborhood of $\zg^*$ for every $\zg\in
\cC$.

We put one more restriction on $\ze$ in this paragraph.  Condition 5
implies that $\ze$ may be chosen so that the following holds.  Let
$\zg$ be any element of $\cC$ which has an endpoint in $\partial D_p$
for some $p\in P_f^\infty $.  Let $x$ be the point of $\zg$ in
$\partial F_p$, and let $z$ be the endpoint of $\zg$ not in
$\overline{F}_p$.  Let $T$ be a path connected set in $S^2\setminus
\{z\}$ with diameter less than $\ze$ which contains a point of $\zg$
and a point $y\in \partial F_p$.  Let $\zr_x$ and $\zr_y$ be rays in
$\overline{F}_p$ which land at $x$ and $y$, respectively.  Then the
union of $\zr_x$, $\zr_y$, $T$ and $\zg\setminus \{z\}$ does not
separate $z$ from an element of $P_f$.  This completes our choice of
$\ze$.

We next choose a positive integer $N$.  Recall that the assumptions
and the choice of $\zr$ provide finitely many open sets $U_i$ which
cover the Julia tiles of $S$ such that the diameters of the connected
components of the sets $f^{-n}(U_i)$ tend to 0 as $n\to \infty $.  We
may refine this collection to a finite collection of open sets $V_j$
such that every $V_j$ is contained in some $U_i$, the $V_j$'s cover
the Julia tiles of $S$ and every $V_j$ intersects every Julia tile of
$S$ in a connected set.  It follows that there exists a positive
integer $N$ such that the diameter of every Julia tile in $f^{-n}(S)$
is less than $\ze$ for every integer $n\ge N$.

We fix an integer $n\ge N$.  We will prove that $f^n$ is isotopic to
the subdivision map of a finite subdivision rule $\cR$.  

Now we define $S_\cR$.  Let $p\in P_f^\infty $.  The disk $D_p$ is
defined so that $D_p=\zj_p(\bbD_\zr)$.  We subdivide $\bbD_\zr$ using
radii and concentric circles as in Figure~\ref{fig:drho}, in general
with many radii and concentric circles, not necessarily uniformly
spaced.  We do this so that the number of concentric circles is
independent of $p$ and if an arc in $\cB\cup \cC$ has an endpoint
$x\in \partial D_p$, then $\zj_p^{-1}(x)$ is not a vertex of this
subdivision of $\bbD_\zr$.  These subdivisions of $\bbD_\zr$, one for
every $p\in P_f^\infty $, induce subdivisions of $S$ and $f^{-n}(S)$.
The Julia tiles in $f^{-n}(S)$ have diameters less than $\ze$.  We
choose these subdivisions of $\bbD_\zr$ so that every Fatou tile in
the induced subdivision of $f^{-n}(S)$ also has diameter less than
$\ze$.  Thus every tile in the induced subdivision of $f^{-n}(S)$ has
diameter less than $\ze$.  We put one more restriction on these
subdivisions.  By possibly including more radii in these subdivisions,
we may assume that there exists a positive integer $M$ such that the
following condition holds.  If $p\in P_f^\infty $, if $\zg_1,\zg_2\in
\cB\cup \cC$ such that $\zg_i$ has an endpoint $x_i\in \partial D_p$
and if one of the two arcs $\za\subseteq \partial D_p$ whose endpoints
are $x_1$ and $x_2$ meets no other arcs in $\cB\cup \cC$, then $\za$
contains exactly $M$ vertices of the subdivision of $D_p$.  The
subdivision of $S$ obtained in this way is $S_\cR$.

  \begin{figure}
\centerline{\includegraphics{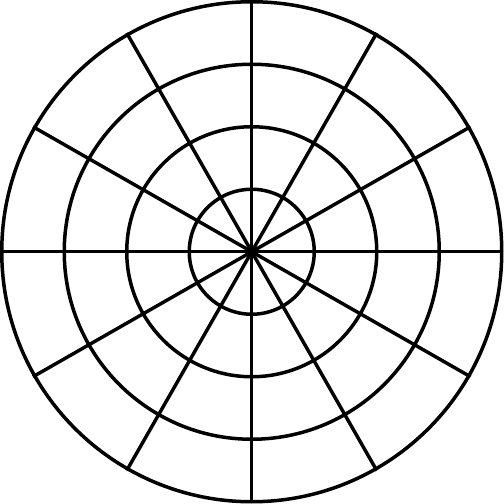}}
\caption{Subdividing $\bbD_\zr$}
\label{fig:drho}
  \end{figure}

It remains to prove that the 1-skeleton of $S_\cR$ is isotopic rel
$P_f$ to a subcomplex of the 1-skeleton of $f^{-n}(S_\cR)$.

We begin the construction of an isotopy by moving only points near the
arcs in $\cB$.  Let $p\in P_f^1$.  Let $\zb$ be an arc in $\cB$ which
contains $p$.  Then $\zb$ is a subarc of a ray in some
$\overline{F}_q$ such that $p\in \partial F_q$.  Moreover, $f^n(\zb)$
is a subarc of a ray in $\overline{F}_{f^n(p)}$, and $f^n(\zb)$
contains a unique arc $\zb'\in \cB$.  So $\zb'$ lifts via $f^n$ to a
subarc $\zg$ of $\zb$.  So there exists a point pushing isotopy taking
$\zb$ to $\zg$ which fixes the complement of $F_q$ and even the
complement of a small neighborhood of $\zb$ which avoids $q$.  Because
$p$ is the only element of $P_f$ in this neighborhood, this isotopy
fixes $P_f$.  Such isotopies, one for every arc in $\cB$, can be
combined into one isotopy which moves the arcs in $\cB$ into the
1-skeleton of $f^{-n}(S_\cR)$.

Now we prepare to move points near the arcs in $\cC$.  We define for
every $p\in P_f^0\cup P_f^\infty $ a subcomplex $W_p$ of
$f^{-n}(S_\cR)$ which is a closed topological disk with $p$ in its
interior.  If $p\in P_f^\infty $, then $W_p$ is the connected
component of $f^{-n}(\overline{D}_{f^n(p)})$ which contains $p$.  Then
$D_p\subseteq W_p\subseteq F_p$ and the boundary of $W_p$ is an
equipotential in $F_p$.  Now suppose that $p\in P_f^0$.  Let $W_p$ be
the smallest closed topological disk such that $W_p$ contains every
tile of $f^{-n}(S_\cR)$ which meets $\overline{B}_\zd(p)$.  Since the
diameters of the tiles of $f^{-n}(S_\cR)$ are less than $\ze<\zd$, we
have that $B_\zd(p)\subseteq W_p\subseteq B_{2\zd}(p)$.

Let $\zg\in \cC$.  We begin to define an arc $\widehat{\zg}$ in the
1-skeleton of $f^{-n}(S_\cR)$ which is isotopic to $\zg$.  Let $x$ be
an endpoint of $\zg$.  If $x\in P_f^0$, then set $p_1=x$.  If $x\in
\partial D_p$ for some $p\in P_f^\infty $, then set $p_1=p$.  We
define $p_2$ in the same way using the other endpoint of $\zg$.  Let
$W$ be the subcomplex of $f^{-n}(S_\cR)$ which as a set is the union
of all tiles of $f^{-n}(S_\cR)$ which meet $\zg$ but are not contained
in either $W_{p_1}$ or $W_{p_2}$.  The choices of $\ze$ and $N$ imply
that $W\subseteq D_\zg$.  The complexes $W_{p_1}$ and $W_{p_2}$ are
constructed so that $\zg\cap \partial W_{p_1}\ne \emptyset $ and
$\zg\cap \partial W_{p_2}\ne \emptyset $.  So the 1-skeleton of $W$
contains a minimal arc $\widehat{\zg}$ which joins $\partial W_{p_1}$
and $\partial W_{p_2}$.  If $i\in \{1,2\}$ and $p_i\in P_f^\infty $,
then $\widehat{\zg}$ meets $\partial F_{p_i}$ in exactly one point.

Suppose that $i\in \{1,2\}$ and that $p_i\in P_f^0$.  We next use
Lemma~\ref{lemma:tileddisk} to construct an extension of
$\widehat{\zg}$ to $p_i$.  We take the complex $X$ in
Lemma~\ref{lemma:tileddisk} to be $W_{p_i}$.  Let $v=p_i$.  Let $u_1$
be the unique vertex in $\partial W_{p_i}\cap \widehat{\zg}$.  The
vertex $u_2$ is gotten in the same way using the other arc in $\cC$
which contains $p_i$.  We use the arc in Lemma~\ref{lemma:tileddisk}
to extend $\widehat{\zg}$ to an arc in the 1-skeleton of
$f^{-n}(S_\cR)$ with endpoint $p_i$.  This completes the definition of
$\widehat{\zg}$.

Now that we have $\widehat{\zg}$, we wish to construct an isotopy rel
$P_f$ which moves $\zg$ to $\widehat{\zg}$.  Moreover, if $p_i=p\in
P_f^\infty $ for either $i=1$ or $i=2$, then we wish to be able to
extend this isotopy to the part of the 1-skeleton of $S_\cR$ which is
in $D_p$.  To do this, we must deal with the difficulty that $\partial
F_p$ might be complicated near $\zg\cap \partial F_p$.  In particular,
the intersection of $F_p$ and the $\ze$-neighborhood of $\zg$ might
not be connected, as in Figure~\ref{fig:gammaist}.  Both $\zg$ and
$\widehat{\zg}$ are drawn as thick arcs in Figure~\ref{fig:gammaist}.
The $\ze$-neighborhood of $\zg$ is shaded dark gray.  The rest of
Figure~\ref{fig:gammaist} will be explained in the next paragraph.

  \begin{figure}
\centerline{\includegraphics{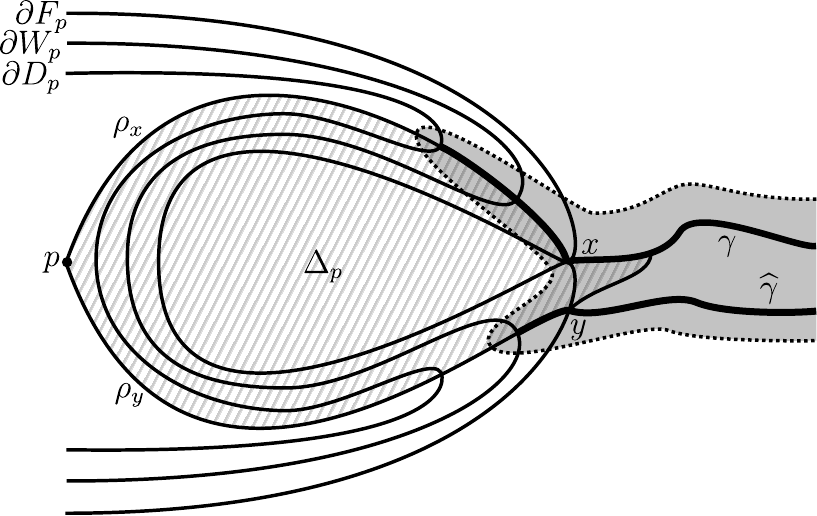}}
\caption{Constructing an isotopy which moves $\zg$ to $\widehat{\zg}$}
\label{fig:gammaist}
  \end{figure} 

With an eye toward Figure~\ref{fig:gammaist}, we proceed as follows.
We continue to assume that $p=p_i\in P_f^\infty $.  Let $x$ and $y$ be
the points at which $\zg$ and $\widehat{\zg}$ meet $\partial F_p$.
Let $\zr_x$ and $\zr_y$ be the rays in $\overline{F}_p$ which land at
$x$ and $y$ such that $\zg$ contains the subarc of $\zr_x$ which joins
$x$ and $\overline{D}_p$, while $\widehat{\zg}$ contains the subarc of
$\zr_y$ which joins $y$ and $W_p$.  The latter subarc is in the
boundary of a tile $t$ of $W$.  We are in the situation of
Lemma~\ref{lemma:canisotop} with $T_\ze=t$ and $y_\ze=y$.  Now we use
the fact that $\zg$ satisfies condition 5.  It follows that there
exists a closed topological disk $\zD_p\subseteq S^2$ whose interior
contains no element of $P_f$ and whose boundary is contained in the
union of $\zr_x$, $\zr_y$, $\zg$ and $t$.  Such a disk is shaded by
line segments in Figure~\ref{fig:gammaist}.

Now we choose an isotopy rel $P_f$ which moves $\zg$ to
$\widehat{\zg}$.  We choose it so that points in the middle section of
$\zg$ move within the smallest closed topological disk in $D_\zg$
which contains $W$.  At an end of $\zg$ near $p=p_i\in P_f^\infty $,
points of $\zg$ move within $\zD_p$.  At an end of $\zg$ near
$p=p_i\in P_f^0$, points of $\zg$ move within $B_{2\zd}(p)$.  We do
this so that if $p=p_i\in P_f^\infty $, then $\zr_x$ moves to $\zr_y$,
$D_p$ remains within $W_p$ and $W_p$ remains within $F_p$.  These
isotopies can be constructed compatibly with supports disjoint from
the arcs in $\cB$.  We obtain an isotopy rel $P_f$ which moves the
arcs in $\cB$ and $\cC$ into the 1-skeleton of $f^{-n}(S_\cR)$.

We finally extend this isotopy to the entire 1-skeleton of $S_\cR$.
Let $p\in P_f^\infty $.  Let $X$ be the intersection of
$\overline{F}_p$ with the union of the arcs in $\cB\cup \cC$.  We have
an isotopy which moves the arcs in $\cB$ and $\cC $ to arcs comprising
sets which we denote by $\cB'$ and $\cC'$.  It also moves the cell
complex $\overline{D}_p$ to a cell complex $\overline{D}'_p \subseteq
W_p$.  Let $X'$ be the intersection of $\overline{F}_p$ with the union
of the arcs in $\cB'\cup \cC'$.  We wish to construct an isotopy which
fixes $p$, $S^2\setminus F_p$ and $X'$ and which moves the 1-skeleton
of $\overline{D}'_p$ into the 1-skeleton of $W_p$.
Figure~\ref{fig:isotopd} shows what we have from the point of view of
$\overline{\bbD}$.  It shows a portion of $\zj_p^{-1}(\partial W_p)$
and a portion of the cell structure of $\zj_p^{-1}(\overline{D}'_p)$.
The cell structure of $\zj^{-1}_p(W_p)$ is not shown.  Suppose that
$\zj_p^{-1}(\zg_1)$ and $\zj^{-1}_p(\zg_2)$ meet $\bbD_\zr$ in points
$x_1$ and $x_2$.  By construction, $x_1$ and $x_2$ are not vertices of
$\bbD_\zr$.  Let $x'_1$ and $x'_2$ be the points in
$\zj^{-1}_p(\partial W_p)$ such that the isotopy which we have moves
$\zj_p(x_1)$ and $\zj_p(x_2)$ to $\zj_p(x'_1)$ and $\zj_p(x'_2)$.  So
$\zj_p(x'_1)$ and $\zj_p(x'_2)$ are endpoints of two arcs in $X'$.
Suppose that one of the two arcs $\za$ in
$\zj^{-1}_p(\overline{D}'_p)$ whose endpoints are $x'_1$ and $x'_2$
contains no other points of $\zj^{-1}_p(X')$.  Let $\za'$ be the
corresponding arc in $\zj^{-1}_p(\partial W_p)$ with endpoints $x'_1$
and $x'_2$.  We wish to construct an isotopy which moves $\za$ to
$\za'$, which fixes $x'_1$ and $x'_2$ and moves every vertex of
$\zj^{-1}_p(\overline{D}'_p)$ in $\za$ to a vertex of
$\zj^{-1}_p(W_p)$.  In order to do this, all that is needed is that
the number of vertices of $\zj^{-1}_p(W_p)$ in $\za'$ is at least as
large as the number of vertices of $\zj^{-1}_p(\overline{D}'_p)$ in
$\za$.  This is guaranteed by the fact that the subdivisions of
$\bbD_\zr$ were chosen so that $\za$ has $M$ vertices and the number
of vertices in $\za'$ is a positive multiple of $M$.  This allows us
to isotop the radii of $\overline{D}'_p$ into the radii of $W_p$ while
fixing $X'$ and $S^2\setminus F_p$.  It is easy to isotop the
equipotentials in $\overline{D}'_p$ to the equipotentials in $W_p$
because they are equal in number.

  \begin{figure}
\centerline{\includegraphics{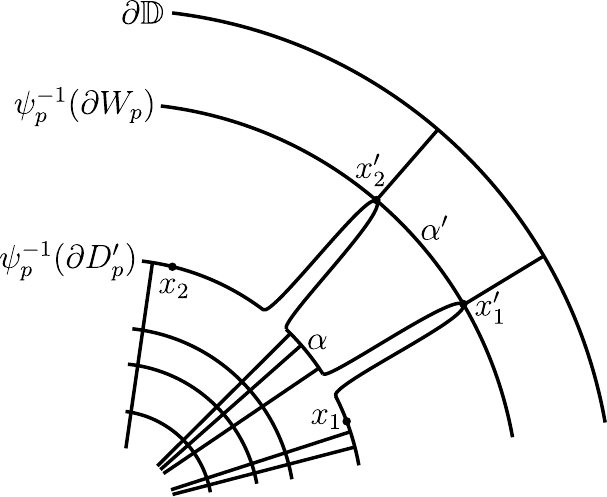}}
\caption{The picture in $\overline{\bbD}$ after moving the arcs in
$\cB$ and $\cC$ } \label{fig:isotopd}
  \end{figure}

Thus the isotopy rel $P_f$ which moves the arcs in $\cB$ and $\cC$
into the 1-skeleton of $f^{-n}(S_\cR)$ can be extended to an isotopy
rel $P_f$ which moves the 1-skeleton of $S_\cR$ into the 1-skeleton of
$f^{-n}(S_\cR)$.  This proves Theorem~\ref{thm:main}.

\end{proof}

\section{Realizability of $(2,2,2,2)$-Euclidean Thurston
maps}\label{sec:euclidean}\nosubsections

This section treats the case when the orbifold of $f$ has signature $(2,2,2,2)$, i.e. is isotopic to a map doubly-covered by an affine torus map. Not every such map $f$ has an iterate isotopic to the subdivision map of a finite subdivision rule; Theorem \ref{thm:euclid} gives a characterization in terms of the eigenvalues of the associated linear map.
 
Our nonrealizability results in this section and the next are
essentially based on the following statement.  If a Thurston map $f$
is the subdivision map of a finite subdivision rule, then the lift of
$f^{-1}$ to the universal covering orbifold of $f$ is combinatorially
distance nonincreasing.  This is formally stated in statement 1 of
Lemma 6.1 of \cite{exppropi}.  We use this result in this section, and
in the next section we use its generalization to the universal
covering space of the complement in $S^2$ of the postcritical set of
$f$.  We discuss this result in the next paragraph.

Let $f\co S^2\to S^2$ be a Thurston map which is the subdivision map
of a finite subdivision rule.  Let $P_f$ be the postcritical set of
$f$.  Let $D$ be either the universal covering orbifold of $f$ or the
universal covering space of $S^2-P_f$.  Let $\zp\co D\to S^2$ be the
associated branched covering map.  The multifunction $f^{-1}$ lifts to
a genuine function $F\co D\to D$, so that $f\circ \zp \circ F=\zp$.
Because $f$ is the subdivision map of a finite subdivision rule, there
exists a cell structure on $S^2$ and a refinement of it called its
first subdivision such that $f$ maps interiors of cells of the first
subdivision homeomorphically to interiors of cells of the initial cell
structure.  In general the original cell structure on $S^2$ does not
lift to $D$, but only because of the need for some vertices ``at
infinity''.  So we enlarge $D$ to a space $D^*$ by adding appropriate
vertices at infinity.  We then construct a cell structure on $D^*$ by
using $\zp$ to lift the initial cell structure on $S^2$.  As a result,
$\zp$ extends to a branched covering map $\zp\co D^*\to S^2$.  The
first subdivision of this cell structure on $D^*$ is the lift of the
first subdivision of the initial cell structure on $S^2$.  The map $F$
also extends to $D^*$.  It homeomorphically maps interiors of cells of
the initial cell structure on $D^*$ to interiors of cells of its first
subdivision.  A fat path in $D^*$ is the union of all tiles which meet
a given curve in $D$ (not $D^*$).  The length of a fat path is 1 less
than the number of these tiles.  The fat path distance function on $D$
(not $D^*$) is defined so that the ``distance'' between points $x,y\in
D$ is the minimum length of a fat path joining them.  (The fat path
function does not define a metric because points in the interior of a
tile have zero distance from each other, and it doesn't define a
pseudometric because a point in an edge or a vertex that is in more
than one tile has positive distance from itself.) Because $F$ maps
initial tiles of $D^*$ into initial tiles of $D^*$, it is distance
nonincreasing with respect to this fat path distance function.

The main result of \cite{lattes} is that almost every Latt\`{e}s
rational map is the subdivision map of a finite subdivision rule with
one tile type of a very special form.  The following theorem treats a
larger class of Thurston maps which need not be rational.  It almost completely determines which
$(2,2,2,2)$-Euclidean Thurston maps are subdivision maps of finite
subdivision rules.  In this case, the orbifold universal cover is
$\mathbb{R}^2$ and the map $F$ in the statement of the theorem is the
inverse of the map $F$ of the previous paragraph.

\begin{thm}\label{thm:euclid} 
Suppose $f$ is a Thurston map whose orbifold $\cO_f$ has signature $(2, 2, 2, 2)$. 
Furthermore, supposed $f$ is normalized so that it has a lift to the plane of the form $F(x)=Ax+b$, where $A$ is a $2\times 2$ matrix
of integers and $b$ is an integral linear combination of the columns of $A$. 
\begin{enumerate}
  \item If both eigenvalues of $A$ have absolute value greater
than 1, then every sufficiently large iterate of $f$ is the
subdivision map of a finite subdivision rule.
  \item If $\pm 1$ is an eigenvalue of $A$, then $f$ is the
subdivision map of a finite subdivision rule with one tile type.
  \item If $A$ has an eigenvalue with absolute value less than 1, then
no iterate of $f$ is Thurston equivalent to the subdivision map of a
finite subdivision rule.
\end{enumerate}
\end{thm}

\noindent{\bf Remark:} Theorem 3.1 of \cite{lattes} implies that in the setting of Theorem \ref{thm:euclid}, if in addition $f$ is rational, then some iterate of $f$ is the subdivision map of a subdivision rule with one tile type. The proof given there does not use the fact that $f$ is rational, only that it is expanding.

  \begin{proof} We begin by making the connection between $A$ and $f$
precise.  Let $\zL$ be the sublattice of $\mathbb{Z}^2$ generated by
the columns of $A$.  Let $\zG$ be the group of all isometries of
$\mathbb{R}^2$ generated by rotations of order 2 about the elements of
$\zL$.  The map $F$ respects the action of $\zG$ on $\mathbb{R}^2$
and induces a map on $\mathbb{R}^2/\zG$ in a canonical way.  The
latter space is homeomorphic to $S^2$.  We conjugate this map on
$\mathbb{R}^2/\zG$ to $S^2$ by this homeomorphism.  The result is $f$.
The image of $\zL$ in $S^2$ is the postcritical set $P_f$ of $f$, and the
image of $\mathbb{Z}^2\setminus \zL$ in $S^2$ is the set of critical
points of $f$. Since the postcritical set is invariant under $f$, $f(P_f)=P_f$, so $F(\zL)\subset \zL$.  In particular $F(0)\in\zL$ and so $b \in \zL$. This verifies that the normalization condition can always be achieved.

Statement 1 follows from Theorem~\ref{thm:main}.

To prove statement 2, suppose that $\pm 1$ is an eigenvalue of $A$.
Replacing $A$ by $-A$ does not change $f$, so we assume that 1 is an
eigenvalue of $A$.  Let $d$ be the degree of $f$.  Since the product
of the eigenvalues of $A$ is $\det(A)=d$, the eigenvalues of $A$ are 1
and $d$.

We continue by finding a normal form for $A$.  Since the eigenvalues
of $A$ are rational numbers, $A$ has nonzero eigenvectors with
rational entries, hence with integer entries.  Let $(p,q)$ be an
eigenvector of $A$ with eigenvalue $d$ such that $p$ and $q$ are
relatively prime integers.  Then there exist integers $r$ and $s$ such
that $ps-qr=1$.  So $(p,q)$ and $(r,s)$ form a basis of
$\mathbb{Z}^2$.  Since $(p,q)$ is an eigenvector with eigenvalue
$d=\text{det}(A)$, when we conjugate $A$ to this basis, we obtain a
matrix of the form $\left[\begin{smallmatrix}d & c \\ 0 & 1
\end{smallmatrix}\right]$.  Since conjugation of $A$ by an element of
$\text{SL}(2,\mathbb{Z})$ corresponds to conjugation of $f$ by a
homeomorphism, we may assume that $A=\left[\begin{smallmatrix}d & c \\
0 & 1 \end{smallmatrix}\right]$.  (Conjugating this by a matrix of the
form $\left[\begin{smallmatrix}1 & a \\ 0 & 1
\end{smallmatrix}\right]$, we may even assume that $0\le c\le d-2$.)

Let $P$ be the parallelogram in $\mathbb{R}^2$ bounded by the lines
given by $y=0$, $y=1$, the 1-eigenspace of $A$ and the line parallel
to this through $(2d,0)$.  See Figure~\ref{fig:egvone}, where dots
mark elements of $\mathbb{Z}^2$ and large dots mark elements of $\zL$.
This parallelogram is a fundamental domain for $\zG$.

  \begin{figure}
\centerline{\includegraphics{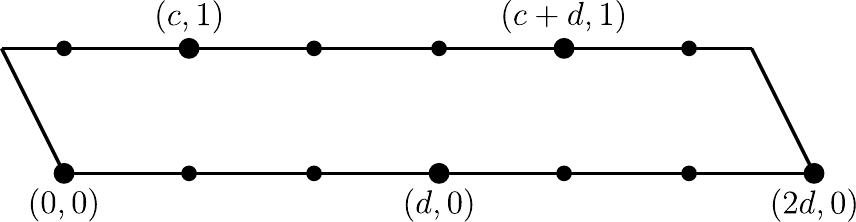}}  \caption{A fundamental
domain $P$ for $\zG$ in the situation of statement 2.}
\label{fig:egvone}
  \end{figure}

The image in $S^2$ of the bottom of $P$ is an arc $\za$ joining the
postcritical points which are the images of $(0,0)$ and $(d,0)$.
Similarly, the image in $S^2$ of the top of $P$ is an arc $\zb$
joining the postcritical points which are the images of $(c,1)$ and
$(c+d,1)$.  The map $F$ maps horizontal lines through elements of
$\mathbb{Z}^2$ to horizontal lines through elements of $\mathbb{Z}^2$.
Hence $f$ stabilizes the union $\za\cup \zb$.  The annulus in $S^2$
determined by $\za$ and $\zb$ is fibered by arcs which are the images
of the line segments in $P$ parallel to the left and right sides of
$P$.  The map $f$ stabilizes the set of these arcs.
The map $f$ also stabilizes
the simple closed curve in $S^2$ which is the image of the line in
$\mathbb{R}^2$ given by $y=1/2$.  The degree of $f$ on this simple
closed curve is $d>1$.  Hence the Lefschetz number of $f$ restricted
to this simple closed curve is not 0, and so the Lefschetz fixed point
theorem implies that $f$ fixes a point of this simple closed curve.
Thus there exists an arc $\zg$ joining $\za$ and $\zb$ which is
stabilized by $f$.  The union of the arcs $\za$, $\zb$ and $\zg$ is a
tree containing the postcritical points of $f$ which is stabilized by
$f$.  This tree provides $S^2$ with the structure of a CW complex
making $f$ the subdivision map of a finite subdivision rule with one
tile type.  This proves statement 2.

To prove statement 3, we argue by contradiction.  Since the hypotheses
are preserved under passing to an iterate, we may replace $f$ by an
iterate and suppose that $A$ has an eigenvalue $\zl$ less than 1 in
absolute value and that $f$ is Thurston equivalent to the subdivision
map $g\co S^2\to S^2$ of a finite subdivision rule.  Since the product
of the eigenvalues of $A$ is the degree of $f$, the eigenvalues of $A$
are real and the eigenvalue other than $\lambda$ has absolute value
greater than 1.  Since multiplying $A$ by $-1$ does not affect $f$, we
may furthermore assume that $\lambda>0$.

Because $f$ and $g$ are Thurston equivalent, there exist two
homeomorphisms $h,h'\co(S^2,P_f)\to(S^2,P_g)$ such that $fh = h'g$
with $h$ and $h'$ isotopic rel $P_f$.  By conjugating $g$, we may
assume that $h' = 1$, so that $f$ and $g$ have the same postcritical
set: $P_f=P_g$.  Then $h(P_f)=P_f$, and $h$ is isotopic to the
identity map rel $P_f$.  Just as $f$ lifts to $F\co \mathbb{R}^2\to
\mathbb{R}^2$ via the branched covering map from $\mathbb{R}^2$ to
$S^2$, the map $h$ has a lift $H\co \mathbb{R}^2\to \mathbb{R}^2$.
Because $h$ is isotopic to the identity map rel $P_f$, we may choose
$H$ so that its restriction to $\zL$ is the identity map.
Furthermore, because $S^2$ is compact, there exists a bound on the
distances which $H$ moves the points of $\mathbb{R}^2$.

The results of the previous paragraph imply that the map $G=F\circ H$
lifts $g$.  So $G^{-1}=H^{-1}\circ F^{-1}$ lifts the multifunction
$g^{-1}$.  Furthermore, there exists a positive real number $J$ such
that $\left\|G^{-1}(x)-F^{-1}(x)\right\|\le J$ for every $x\in
\mathbb{R}^2$.  We choose $J$ so large that we even have that
$\left\|G^{-1}(x)-A^{-1}x\right\|\le J$ for every $x\in \mathbb{R}^2$.

In this paragraph we prove that there exists a positive real number
$K$ such that $\left\|A^{-n}w\right\|\le K \zl^{-n}\left\|w\right\|$
for every positive integer $n$ and every $w\in \mathbb{R}^2$.  If $A$
is diagonal, then this is clear.  In general, $A$ can be conjugated by
an element of $\text{GL}(2,\bbR)$ to a diagonal matrix.  Since the
conjugating matrix deforms the standard metric on $\bbR^2$ by a
bounded amount, the desired result follows.

Now let $L$ be the $\lambda$-eigenspace of $A$.  We will obtain a
contradiction by considering the action of the iterates of $G^{-1}$ on
$L$.  Let $x\in L$.  Then $G^{-1}(x)=A^{-1}x+y_1=\lambda^{-1} x+y_1$
for some $y_1\in \mathbb{R}^2$ with $\left\|y_1\right\|\le J$.
Similarly, $G^{-2}(x)=\lambda^{-2}x+A^{-1}y_1+y_2$ for some $y_2\in
\mathbb{R}^2$ with $\left\|y_2\right\|\le J$.  Using the result of the
previous paragraph, we inductively see for every positive integer $n$
that
  \begin{equation*}
\begin{aligned}
\left\|G^{-n}(x)-\lambda^{-n}x\right\| & \le J K(\lambda^{-n+1}+\cdots
+\lambda^{-2}+\lambda^{-1}+1)\\
 & = JK(\lambda^{-n}-1)(\lambda^{-1}-1)^{-1}.
\end{aligned}
  \end{equation*}
Hence
  \begin{equation*}
\begin{aligned}
\lambda^{-n}\left\|x\right\| 
&=\left\|\lambda^{-n}x-\lambda^{-n}0\right\|\\
&=\left\|(\lambda^{-n}x-G^{-n}(x))+(G^{-n}(x)-G^{-n}(0))
+(G^{-n}(0)-\lambda^{-n}0)\right\|\\
 & \le
\left\|\lambda^{-n}x-G^{-n}(x)\right\|+\left\|G^{-n}(x)-
G^{-n}(0)\right\|+\left\|G^{-n}(0)-\lambda^{-n}0\right\|\\
 & \le 2JK(\lambda^{-n}-1)(\lambda^{-1}-1)^{-1}
+\left\|G^{-n}(x)-G^{-n}(0)\right\|,
\end{aligned}
  \end{equation*}
and so
  \begin{equation*}
\begin{aligned}
\left\|G^{-n}(x)-G^{-n}(0)\right\| & \ge
\lambda^{-n}\left\|x\right\|-2JK(\lambda^{-n}-1)(\lambda^{-1}-1)^{-1}\\
 & \ge(\left\|x\right\|-2JK(\lambda^{-1}-1)^{-1})\lambda^{-n}.
\end{aligned}
  \end{equation*}
Thus if $\left\|x\right\|>2JK(\lambda^{-1}-1)^{-1}$, then $\left\|
G^{-n}(x)-G^{-n}(0)\right\|$ tends to $\infty$ as $n$ tends to $\infty$.

We finally use the fact that $g$ is the subdivision map of a finite
subdivision rule.  Note that $G^{-1}$ is a lift of $g^{-1}$ to the
universal covering orbifold of $g$.  So the discussion at the
beginning of this section applies.  The initial tiling of $S^2$ lifts
to a tiling $T$ of $\mathbb{R}^2$.  Because $S^2$ is compact and this
branched covering map is regular, the tiles of $T$ decompose into
finitely many orbits under the action of its group of deck
transformations.  Because the elements of this group are Euclidean
isometries, there is thus a bound on the diameters of the tiles of
$T$.  So a bound on the fat path distance between two points of
$\mathbb{R}^2$ provides a bound on the Euclidean distance between
these points.  Hence the lack of a bound on the Euclidean distances
$\left\|G^{-n}(x)-G^{-n}(0)\right\|$ implies the lack of a bound on
the fat path distances between $G^{-n}(x)$ and $G^{-n}(0)$.  This
contradicts the observation made at the beginning of this section that
$G^{-1}$ is distance nonincreasing with respect to the fat path
distance function.

This proves Theorem~\ref{thm:euclid}.

\end{proof}

\section{Nonrealizability conditions}\label{sec:nofsr}\nosubsections

Suppose $f$ is a Thurston map with postcritical set $P_f$.  
We denote by $\cC$ the set of simple, unoriented, essential,
non-peripheral (not homotopic into arbitrarily small neighborhoods of
elements of $P_f$) curves in $S^2-P_f$, up to homotopy in
$S^2-P_f$. We denote by $\nonslope$ the union of homotopy classes of
inessential and peripheral curves in $S^2-P_f$. Via pullback, we
obtain a \emph{pullback relation} $\pullback$ on
$\cC\cup\{\nonslope\}$, where $\gamma \pullback \wtgamma$ if and only
if some component $\delta \subset f^{-1}(\gamma)$ is in $\wtgamma$.
It is natural to view this as a multivalued map of
$\cC\cup\{\nonslope\}$ to itself, and to consider its orbits; a point
has finite image under this map.  Note that the image of $\nonslope$
under this multivalued map consists of only $\nonslope$. An orbit
$\gamma_0\pullback\gamma_1\pullback\gamma_2\ldots$ is \emph{univalent}
if the unsigned degrees $\deg(f: \gamma_{i} \to \gamma_{i-1}), i=1, 2,
\ldots$ are all equal to one; it is \emph{wandering} if the
$\gamma_i$'s are all distinct and not equal to $\nonslope$.

Similarly, we denote by $\cA$ the set of homotopy classes of arcs $\alpha: [0,1] \to S^2$ with $\alpha(\{0,1\}) \subset P_f$, $\alpha(t) \in S^2-P_f$ for $t \not\in \{0,1\}$, up to homotopy through arcs with the same properties and reparameterizations which may reverse orientation, and subject to the condition that $\alpha$ is not peripheral in a similar sense.  Note that an element of $\cA$ might join a point of $P_f$ to itself.  We denote by $\nonarc$ the union of those homotopy classes of arcs $\alpha: [0,1] \to S^2$ defined similarly as for $\cA$, but with $P_f$ replaced by $f^{-1}(P_f)$, and subject to the condition that they do \emph{not} define elements of $\cA$. Again by pullback we obtain a relation $\cA \cup \{\nonarc\} \pullback \cA \cup \{\nonarc\}$ such that the corresponding multivalued map has finite images of points, and $\nonarc$ maps to itself. The corresponding degrees by which arcs map are always equal to one. Wandering is defined similarly.

Theorem~\ref{thm:cnds} provides conditions under which no iterate of a
given Thurston map is Thurston equivalent to the subdivision map of a
finite subdivision rule.  

\begin{thm}\label{thm:cnds} Let $f\co S^2\to S^2$ be a Thurston map
with postcritical set $P_f$.  
%
If the pullback relation on arcs has wandering orbits, or if the pullback relation on curves has wandering univalent orbits, then no iterate of $f$ is
Thurston equivalent to the subdivision map of a finite subdivision
rule.
\end{thm}

\noindent{\bf Remark:} If $f$ is homotopic to a B\"ottcher expanding map, then (i) its pullback relation on arcs has no wandering orbits, and (ii) its pullback relation on curves has no univalent wandering orbits.  Compare \cite{P2}. For maps without periodic branch points, similar and stronger results are shown in \cite[Prop. 5, Thm. 8]{HP2}.  
Fix some length metric as in the definition of B\"ottcher expanding. 

To prove (ii), note that the length of any essential, simple, nonperipheral curve is bounded from below away from zero; that there are only finitely many such curves whose length is bounded above by some positive constant; and that if $f$ maps such a curve $\widetilde{\gamma}$ by degree $1$ to another such curve $\gamma$ by degree $1$, then the definition of B\"ottcher expanding implies $\ell(\tilde{\gamma}) < c \ell(\gamma)$ where $c<1$.

To prove (i), we argue as follows. For $p \in P_f^\infty$ and $0<\rho<1$ denote by $D_{p, \rho}$ the disk of radius $\rho$ about $p$ in the local B\"ottcher coordinates. Define the \emph{complexity} $C[\alpha]$ of a homotopy class of arcs $\alpha$ to be the infimum, among homotopic representatives, of the length of the portion of the subarc lying outside the disks $D_{p,1/2}$. There are only finitely many arc classes whose complexity is bounded above by some positive constant. B\"ottcher expansion  implies that if $f$ maps such an arc $\widetilde{\alpha}$ to such an arc $\alpha$, then $C[\widetilde{\alpha}]) < cC[\alpha]+\delta$ where $c<1$ is as in the previous paragraph and $\delta>0$ is the maximum, over such $p$, of the distance between $\partial D_{p,1/2}$ and $\partial D_{p, (1/2)^{1/\deg(f)}}$. The point is that under taking preimages, the middle of an arc shortens by a definite multiplicative amount, while the end segments lengthen by a bounded additive amount.  Combined, these two facts mean that there are no wandering arcs under iterated pullback. 

  \begin{proof} Suppose $\alpha_0 \pullback \alpha_1 \pullback \ldots$ is such an orbit. 
  
  Every iterate of $f$ satisfies the assumptions of the
theorem with the sequence $\za_0,\za_1,\za_2,\ldots$ replaced by a
subsequence.  Every Thurston map homotopic to $f$ satisfies the
assumptions of the theorem with $\za_0,\za_1,\za_2,\ldots$ modified by
homotopies rel $P_f$.  Thus to prove the theorem, it suffices to prove
that $f$ is not the subdivision map of a finite subdivision rule.

We prove this by contradiction.  Suppose that $f$ is the subdivision
map of a finite subdivision rule $\cR$.  We modify $\za_0$ by a
homotopy rel $P_f$ so that, except for its endpoints in the case of
arcs, for some nonnegative integer $n$ it meets the 1-skeleton of 
$\cR^0(S^2)$ transversely in $n$ points which are not vertices.  We
then modify $\za_1,\za_2,\ldots$ so as to preserve the assumptions of
the theorem.  Because $f$ is the subdivision map of $\cR$ and $f$ maps
$\za_i$ bijectively to $\za_{i-1}$, it follows that, except for
endpoints in the case of arcs, for every nonnegative integer $i$,
$\za_i$ meets the 1-skeleton of
$\cR^0(S^2)$ transversely in at most $n$ points which are not vertices.

Because $\za_0$, $\za_1$, $\za_2,\ldots$ are mutually not homotopic
rel $P_f$, the order of $P_f$ must be at least 4.  So the universal
covering space of $S^2\setminus P_f$ is the open disk $D$.  Let
$\zp\co D\to S^2\setminus P_f$ be the universal covering map.  As at
the beginning of Section~\ref{sec:euclidean}, we enlarge $D$ to a
space $D^*$ by adding vertices at infinity and extend $\zp$ to a
branched covering map $\zp\co D^*\to S^2$.  We use $\zp$ to lift the
cell structure of $S^2$ to $D^*$.

Let $F$ be a fundamental domain in $D^*$ for $\zp$ which is a union of
tiles of $D^*$.  For every nonnegative integer $i$, let
$\widetilde{\za}_i$ be a lift of $\za_i$ to $D^*$ which contains a
point in the interior of $F$.  Then each $\widetilde{\za}_i$ meets the
interior of $F$ and crosses at most $n$ edges of $D^*$.  Since there
is a positive integer $N$ such that every tile of $D^*$ has at most
$N$ vertices and edges, these lifts are contained in the union of
finitely many translates of $F$ under the group of deck
transformations of $\zp$.  In the case of arcs, it follows that
$\widetilde{\za}_i$ and $\widetilde{\za}_j$ have the same initial and
terminal endpoints for some indices $i\ne j$.  It follows that
$\widetilde{\za}_i$ and $\widetilde{\za}_j$ are homotopic rel
$\partial D^*$.  Hence $\za_i$ and $\za_j$ are homotopic rel $P_f$, a
contradiction.  The case of simple closed curves is only slightly
different.

This proves Theorem~\ref{thm:cnds}.

\end{proof}

Theorem~\ref{thm:euclid} shows that some Thurston maps with
Euclidean orbifolds are not subdivision maps of finite subdivision
rules.  We illustrate Theorem~\ref{thm:cnds} with the following
example of a Thurston map with hyperbolic orbifold such that no iterate
is Thurston equivalent to the subdivision map of a finite
subdivision rule.

\begin{ex}\label{ex:cubic} We begin by describing the subdivision map
of a finite subdivision rule on $S^2$.  The 1-skeleton of the initial
cell structure on $S^2$ is a square with four vertices and four edges.
There are two tiles, each a quadrilateral.  The subdivision map $g$, a
Thurston map, acts as in Figure~\ref{fig:cubic}.
Figure~\ref{fig:cubic} shows two copies of $S^2$, each with one point
at $\infty$.  The initial cell structure is shown in the right part of
Figure~\ref{fig:cubic} with the two tiles labeled $A$ and $B$.  The
map $g$ fixes all four initial vertices and the two horizontal edges
pointwise.  The tile labels in the left part indicate the initial
tiles which are the images of these tiles.  So $g$ has degree 3.  Its
local degree at every critical point is 2.  These four critical points
are all fixed by $g$, and so they form the postcritical set of $g$.

  \begin{figure}
\centerline{\includegraphics{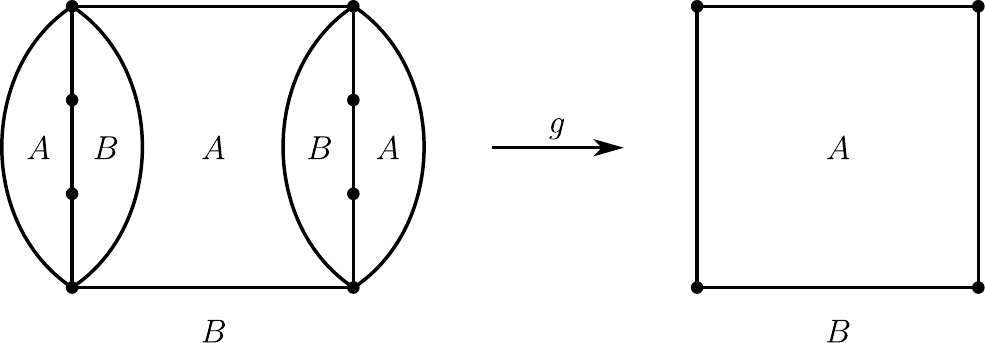}}  \caption{Defining the
Thurston map $g$ for Example~\ref{ex:cubic}.}
\label{fig:cubic}
  \end{figure}

Let $\zt$ be a Dehn twist, or even a half Dehn twist, about a simple
closed curve which is vertical relative to the orientation of
Figure~\ref{fig:cubic}.  Let $f=\zt^n\circ g$ for some nonzero integer
$n$.  Let $\za_0$ be the top edge of tile $A$.  We lift $\za_0$ via
the iterates of $f$ to obtain arcs $\za_1,\za_2,\za_3,\ldots$.  The
arc condition of Theorem~\ref{thm:cnds} is satisfied.  So no iterate
of $f$ is Thurston equivalent to the subdivision map of a finite
subdivision rule.

\end{ex}

The canonical decomposition of a Thurston map \cite[Thm. 10.2]{P2} provides another source of examples. 

\begin{ex} Suppose $f$ is an obstructed Thurston map with an elliptic piece in its decomposition that possesses at least four marked points. Then there exists a homeomorphism $h: (S^2, P_f) \to (S^2, P_f)$ such that the twist $h \circ f$ is not homotopic to the subdivision map of a finite subdivision rule.

For one may arrange so that $h$ is supported on the given elliptic piece $U$, and the first-return map of $h \circ f$ is pseudo-Anosov on $U$, so that every curve contained in this elliptic piece wanders under pullback of the restriction of $f$ to this piece. 

A bit more concretely: the mating of the $1/4$-rabbit quadratic polynomial with its complex conjugate provides an obstructed Thurston map $f_0$; its fourth iterate $f_1=f_0^4$ has the property that it is homotopic to a map which is the identity on a subsurface $U \subset S^2-P_{f_1}$ such that $U$ is a sphere with four holes and the boundary of each hole is essential and nonperipheral in $S^2-P_{f_1}$. Setting $f=h\circ f_1$ where $h: U \to U$ is pseudo-Anosov and $h|_{\partial U}=\id$ will realize this construction. 
\end{ex}

\end{document}